\documentclass[12pt,a4paper,oneside]{amsart}

\usepackage{latexsym,amsfonts,hyperref,euscript,algorithm,algorithmic,setspace}

\swapnumbers
\newtheorem{theorem}{Theorem}[section]

\newtheorem{lemma}[theorem]{Mini Lemma}

\newtheorem*{theorem*}{Theorem}
\newtheorem*{lemma*}{Lemma}
\newtheorem*{fact*}{Fact}
\newtheorem*{claim*}{Claim}
\newtheorem*{ConjLemma*}{Conjugation Lemma}

\theoremstyle{remark}
\newtheorem{remark}[theorem]{Remark}

\newtheorem*{example*}{Example}

\newcommand{\MathRoman}[1]{\mathop{\mathrm{#1}}\nolimits}
\newcommand{\mathcan}[1]{\mathbb{#1}}

\newcommand{\FF}{\mathcan{F}}
\newcommand{\NN}{\mathcan{N}}

\newcommand{\ZZ}{\mathcan{Z}}

\newcommand{\fp}{{\mathbb F}_ p}
\newcommand{\fps}{{\mathbb F}^*_ p}
\newcommand{\fqs}{{\mathbb F}^*_ q}
\newcommand{\GG}{{\bf G}}

\newcommand{\Out}{{\rm Out}}

\newcommand{\f}{\mathbb F}
\newcommand{\Fq}{\mathop{\FF_{q}}\nolimits}

\renewcommand{\mod}{\MathRoman{mod}}

\begin{document}

% Title
% -----

\title{A Short Note on Discrete Log Problem in $\fps$}

\author{Habeeb Syed}

\address{Information Security Group
\newline \indent
Computational Research Laboratories Limited
\newline \indent
Pune,  411 016 - INDIA}
\email{habeeb@crlindia.com}


\begin{abstract}
Let $p$ be a odd prime such  that $2$ is a primitive element of finite
field  $\fp$.  In this  short  note we  propose  a  new algorithm  for the
computation of discrete logarithm in $\fp^*$.  This algorithm is based on elementary properties of finite fields and is
purely theoretical in nature.

\end{abstract}


\maketitle

%\newpage

%\bigskip \bigskip \bigskip \bigskip

%\tableofcontents

%\newpage

% body of the article
% -------------------
\vskip-1cm
\section*{Introduction}
\label{Intro}
Consider a finite field $\Fq$ (also denoted by ${\rm GF}(q)$), where $q=p^r$, $p$ is a prime and $r\in \NN:=\{1,2,3,\ldots \}$.
Let $\alpha$ be a primitive element of $\Fq$ i.e., generator of the multiplicative cyclic group $\fqs$.
For arbitrary  element  $b  \in  \fqs$
computing $n\in \NN, n \leq q-1$ such that
\begin{equation}
 \label{Eqn}
b = \alpha^n\;\mod p
\end{equation}
is known as  \textit{discrete log problem} (DLP) in  $\fqs$. 
Discrete log computation in finite fields is an important problem mainly due to applications of these groups in 
cryptography. Beginning
with  Diffie-Hellman  key  exchange  protocol \cite{DifHel},  El  ElGamal
encryption/signature scheme \cite{ElGamal} the  DLP in $\fqs$ has been
used as  basic mathematical  primitive in many  cryptographic schemes,
and security of these systems  depend  on difficulty of DLP in respective
$\fqs$. It  is rather difficult to  give even reasonably  good list of
references  to   all  the  work   involving  DLP  in   $\fqs,$  however
\cite{Mov,OdlOdl} are good to begin with.

In the last couple of decades DLP in $\fqs$ has been studied extensively
and  several algorithms have  been proposed  for the  computation same.
Most efficient algorithm for the computation of DLP is the one based
on Number Field Sieve \cite{Gordon,Schi}.  See  also
\cite{CopShp,MulWhi}  for   results  which  are   not  computationally
oriented  but  certainly give  insight  into  the  problem of  DLP  in
$\fqs$. In this short note we are focused on odd primes $p$ for which
$2$ is  primitive element of $\fp$.  For such primes we  propose a new
algorithm  to compute  discrete logarithm  in $\fps$.  The proposed
algorithm is based on elementary properties of finite fields and is purely theoretical in nature.
Further,  complexity of the algorithm is exponential, and as such it is not being 
suggested for any computational purposes. 
This short note has two sections. In section \ref{Algo} we begin with
basic results needed  and then explain the algorithm in  detail. In section
\ref{Analysis}  we analyze the  complexity of  the algorithm.

\section{The Algorithm}
\label{Algo}
In  reminder of this note  $p$  denotes   odd  prime  and  $r  \in  \NN$.   By
$\log_\alpha  b=n$  we mean  $n$  as  in  \eqref{Eqn}. We  begin  with
following simple results.

\begin{lemma}
 \label{TheLemma}
Let $a,b  \in \fqs,\;(q=p^r)$ be  such that $a+b=0(\mod p),$  then for
any primitive element $\alpha$ of  $\Fq$ we have,
$$ \log_\alpha  a - \log_\alpha  b = \log_\alpha  b - \log_\alpha  a =
\frac{q-1}{2}\; \mod (q-1).
$$
\end{lemma}
\begin{proof}
For any $a,b \in \fqs$ we have,
$$  a+b=0 (\mod  p) \Longleftrightarrow  \frac{a}{b}=\frac{b}{a}  = -1(\mod p).
$$ Computing discrete logarithm  with respect to any  primitive element
$\alpha$ of $\Fq$,  we have,
$$  \log_\alpha\frac{a}{b} =  \log_\alpha\frac{b}{a}= \log_\alpha  a -
\log_\alpha b = \log_\alpha (-1). 
$$
Now the conclusion follows from the simple observation, 
\begin{equation}
 \label{logofminusone}
\log_\alpha (-1) = \frac{q-1}{2}\; \mod (q-1).
\end{equation} \end{proof}  
\begin{remark}
\label{remark}
The result \eqref{logofminusone} is true in more generality: Let $\GG$ be a finite cyclic
group of even order say, $2m.$ Suppose $\alpha$ is a primitive element of $\GG$. It is easy to see that the element 
$\beta=\alpha^m$ is the only non-trivial element fixed by all automorphisms of $\GG$. This implies
that the discrete logarithm of $\beta$ is independent of primitive element $\alpha$ of $\GG$ and is equal to $m$. In case
of $\GG=\Fq$, we have \eqref{logofminusone}.
\end{remark}

The  proposed  algorithm  depends on  above  lemma  and
following simple fact:
\begin{enumerate}
 \item[{\bf Fact 1.}] Let $a,b\in \NN,\;1 <a,b <p$ be such that $a+b=p$,
   then precisely one of $a,b$ is divisible by $2.$
 
\end{enumerate} 
\smallskip 
 Before  we  explain  the  algorithm  we remind  that  this
algorithm  computes $n$  in  \eqref{Eqn} when $p$ is a odd prime  such that   
  $2$ is a primitive element of $\fp$. A  necessary condition for such  a thing to
happen is that $p \equiv \pm 3 (\mod 8)$ \cite[Chap 4]{Baker}.  Next we explain
the proposed algorithm with the help of simple example.
\begin{example*}
 \label{example2}
Consider  the cyclic  group  $\f_{37}^*$ which  is  generated by  $2$.
Suppose we  want to find $\log_2  3$.  Noting that  all the operations
are  performed $\mod 37,$  the proposed algorithm  works as  follows
\newpage 
We have $3+34=37$ and hence
\begin{align*}
 3=-34=  & 2\cdot(-17)=  2 \cdot  20 \\ 
 =  & 2\cdot(4\cdot5)=2^3\cdot
 5=2^3\cdot(-32)\\ =& 2^3\cdot2^5\cdot(-1)=2^8\cdot2^{18}=2^{26} 
\end{align*}
We have $\log_2(3)=26.$
\end{example*}

%\newpage
Now we are ready to state the algorithm.
%\vskip2cm
\begin{algorithm}
\label{TheAlgo}
\caption{}        \flushleft{INPUT:        Element       $b$        of
  $\fps$}\\ \flushleft{OUTPUT: Discrete Log of $b$ to base $2$}\\
\begin{algorithmic}[1]
\STATE  Initialize $\Out  =  0  $ \IF{$b=1$}  \STATE  return 0  \ENDIF
\WHILE{$b\neq 1$} \STATE Find the max  power $k$ of 2 that divides $b$
\IF{   $k=0   $}  \STATE   $b=   p-b$   \STATE  $\Out=   \Out+(p-1)/2$
($\mod\;(p-1)$)     \ELSE      \STATE     $b=     b/(2^k)$,     $\Out=
\Out+k$($\mod\;(p-1)$) \ENDIF \ENDWHILE
\end{algorithmic}
\end{algorithm}

Next we prove that the Algorithm \ref{Algo} converges.
\begin{proof}
\label{proofofalgo}
Suppose  we  want  to  compute  $\log_2  b\;(  b\in  \NN,  1<b<p)$  in
$\fp^*$.   Let   $b=2^rb',$   $b'$   not  divisible   by   $2$,   then
$\log_2b=r+\log_2b'$, and hence  if needed we can replace  $b$ by $b'$
and assume  that $b$ is not  divisible by $2$.  Since  we are assuming
that $2$ is  primitive element of $\fp^*$, there  exists $t$ such that
$1<t<p-1$ and
\begin{equation}
\label{eqn01} 
b \equiv 2^t (\mod p).
\end{equation}

Let $b_0=p-b$. Since  $b$ is not divisible by $2$  we have that $b_0$
is divisible by  $2$. Let $b_0=2^rb_1$ where $r \in  \NN$ and $b_1$ is
not divisible by $2$.  If $b_1=1$, we are done. Otherwise,

\begin{claim*}
  $r<t$.
\end{claim*}

 \noindent Suppose not; let $r=t+s, s\in \NN$, then from \eqref{eqn01}
 we have,

\begin{equation}
\label{eqn03}
(-b_1)2^r  \equiv 2^t  (\mod p)  \Longrightarrow  p\; \mbox{divides}\;
2^t+b_1 2^{t+s} = 2^t (1+b_1 2^s)
\end{equation}
and hence $p$ divides $(1+b_1 2^s)$.  On the other hand $(1+b_1 2^s) <
b+b_0 =p$  and hence the only way  $p$ can divide $(1+b_1  2^s)$ is if
$1+b_1 2^s=0$ in $\ZZ$, which clearly is not the case.

So we have
\begin{equation}
\label{eqn04}
-b_1 \equiv 2^{t-r} (\mod p).
\end{equation}
Now we  are back  to \eqref{eqn01} with  $b=p-b_1$ and  $t=t-r$. Thus,
after at most $t$ iterations  the algorithm stops and returns value of
$\log_2 b$. \end{proof}
\section{Analysis of  The Algorithm}
\label{Analysis}

Throughout this section $p$ denotes odd primes for which $2$ is primitive element of $\fp.$ For a given $b \in \fps$, to compute $\log_2 b$, Algorithm.\ref{TheAlgo} repeats  steps $(6)-(8)$  each time  replacing $b$ by  $p-b'$ until
$b'=\pm 1(\mod p)$.  The space requirements to execute the algorithm are not significant, but the order of growth of computations  is  ${\rm O}(2^{(p-1)/2})$. This algorithm does not give any advantages over the existing algorithms in terms of complexity.  Our computational experiments with the algorithm suggested that while implementation of the algorithm worst case scenario (in terms of time taken to compute) occurred while calculating $\log_2((p-1)/2)$. However one can easily check,
$$
\log_2(\frac{p-1}{2})= \frac{p-3}{2}
$$

\section*{Acknowledgments} Author is thankful to    Mr. Ramanjulu for his remarks as well
as for his help in coding and computations.

\end{document}